\documentclass[a4paper, 11pt,reqno,oneside]{amsart}
\usepackage[top = 1in, bottom = 1in, left=1.25in, right=1.25in]{geometry}
\usepackage{setspace} 
\usepackage{amssymb,amsmath,amsthm,graphicx,xcolor,mathtools,mathrsfs}
\usepackage{braket} 
\usepackage[pagebackref]{hyperref}
\hypersetup{colorlinks, linkcolor={red!50!black}, citecolor={green!50!black}, urlcolor={blue!50!black}}
\usepackage[shortlabels]{enumitem}
\usepackage{todonotes}
\usepackage{nicefrac}
\usepackage{bm}

\newtheorem{theorem}{Theorem}
\newtheorem{lemma}[theorem]{Lemma}
\newtheorem{corollary}[theorem]{Corollary}

\newtheorem{claim}[theorem]{Claim}

\DeclareMathOperator{\Gallai}{gc}
\DeclarePairedDelimiterX{\abs}[1]{\lvert}{\rvert}{#1}
\newcommand{\Var}{\operatorname{Var}}

\title{Gallai $3$-colourings of random graphs}

\author[F. S. Benevides]{Fabrício S. Benevides}
\address[F. S. Benevides]{Departamento de
Matemática, Universidade Federal do Ceará (UFC), Av. Humberto Monte, Campus do Pici (Bloco 914). 60.455-760, Fortaleza, Brazil.}
\email{fabricio@mat.ufc.br}

\author[R. C. S. Monteiro]{Rubens C. S. Monteiro}
\address[R. C. S. Monteiro]{Departamento de
Matemática, Universidade Federal do Ceará (UFC), Av. Humberto Monte, Campus do Pici (Bloco 914). 60.455-760, Fortaleza, Brazil.}
\email{cainan14@yahoo.com.br}

\author[G. O. Mota]{Guilherme O. Mota}
\address[G. O. Mota]{Instituto de Matemática e
  Estatística, Universidade de São Paulo, Rua do Matão 1010,
  05508-090 São Paulo, Brazil.}
  \email{mota@ime.usp.br}

\thanks{F.\,S.\,Benevides was supported by CNPq (313558/2025-6), FAPESB (EDITAL
		FAPESB No 012/2022 - UNIVERSAL - NºAPP0044/2023) and CAPES Finance
		Code 001. G. O. Mota was supported by CNPq (420838/2025-2 and
		315916/2023-0) and FAPESP (2024/13859-4 and 2023/03167-5).}

\begin{document}

\onehalfspacing
\begin{abstract}
      A Gallai $k$-colouring of a graph $G$ is a colouring of $E(G)$ with
      $k$ colours that induces no rainbow triangles, that is, a triangle with edges of 3 different colours. We give a first step
      towards estimating the number of Gallai colourings of the Erd\H os--Rényi
      random graph, by proving that for every $\delta > 0$ there are $c$
      and $C$ such that with high probability the number of Gallai
      \mbox{$3$-colourings} of $G(n,p)$ is at least $3^{(1-\delta)\binom{n}{2}p}$
      for $p \leq cn^{-1/2}$, and at most $2^{(1+\delta)\binom{n}{2}p}$
      for $p \geq Cn^{-1/2}$.
\end{abstract}

\maketitle

\section{Introduction} \label{sec:intro}

For $k\in \mathbb{N}$, let $[k] = \{1,\ldots, k\}$. A \emph{Gallai $k$-colouring} of a graph $G$ is a colouring of $G$
with colours $[k]$ that does not contain a rainbow
triangle, that is, a triangle whose 3 edges receive distinct
colours. The term \emph{Gallai colouring} was originally used by
Gyárfás and Simonyi \cite{GySi2004edge}, who consider the particular
case of complete graphs. This nomenclature is due to a relation with a
theorem by Gallai presented in one of his most influential articles
(see \cite{gallai1967transitiv,gallai1967-translation}) about
transitively orientable graphs. Colourings free of rainbow triangles
had previously been studied under the name \emph{Gallai partitions} by
K\"{o}rner, Simonyi, and Tuza \cite{korner1992perfect}. In 2010,
Gy\'{a}rf\'{a}s and S{\'a}rk{\"o}zy~\cite{gyarfas2010gallai}
investigated Gallai colourings in non-complete graphs and demonstrated
structural properties that must be satisfied by every Gallai
colouring. For connections of Gallai colourings with Information
Theory and Ramsey Theory, we refer the reader to~\cite{chen2018gallai,
      chua2013gallai, fox2015erdHos, fujita2011gallai,korner2000graph,
      wu2018all,zhang2018gallai}.

In a different context, motivated by problems of various natures and
by a classic problem of Erd\H{o}s and Rothschild (see,
e.g.,~\cite{erdos1974some}), several authors
\cite{alon2004number,benevides2017edge,
      hoppen2017rainbowfull,pikhurko2017erdHos}
have studied counting problems for the number of colourings that avoid
subgraphs (or subhypergraphs) coloured in a predefined way. In 2019,
improving a result in~\cite{bastos2018counting}, the first author,
Han, and Bastos \cite{2019GallaikCol-JCBT}, and independently, Balogh
and Li \cite{balogh2019typical}, showed that almost every Gallai
$k$-colouring of a sufficiently large complete graph uses only~$2$
colours. The result in \cite{balogh2019typical} uses the Container
Method and assumes that $k$ is constant; in
\cite{2019GallaikCol-JCBT}, an ad-hoc proof is given, and the value of
$k$ can be exponentially large in the number of vertices (this is
tight, as the result is false when $k$ is yet exponentially large but
with a larger base).

We take the first step in the study of the number of Gallai colourings of
the Erd\H os--Rényi random graph, $G(n,p)$, by estimating such number when $k=3$. Let $\Gallai(G)$ denote the number of
Gallai $3$-colourings of a graph $G$ and note that for any graph $G$,
we have $2^{e(G)} \le \Gallai(G) \le 3^{e(G)}$, where $e(G) = |E(G)|$. Roughly speaking,
 our main result in this short note shows that with high probability,
w.h.p.\  for short, the number of Gallai $3$-colourings of $G=G(n,p)$ is
``close'' to one of the extremes $2^{e(G)}$ or $3^{e(G)}$,
whenever $p$ is not ``close'' to the threshold $n^{-1/2}$. 

\begin{theorem}\label{thm:mainresult}
      For every~$\delta>0$, there exist~$c$ and~$C$ such that for~$G=G(n,p)$ the following hold (asymptotically) with high probability.
      \begin{enumerate}[(i)]
            \item \label{resultado-parte1} if $p\leq c n^{-1/2}$, then
                  $\Gallai(G)\geq 3^{(1-\delta)e(G)};$\label{eq:main1}
            \item\label{resultado-parte2} if $p\geq C n^{-1/2}$, then
                  $\Gallai(G)\leq 2^{(1+\delta)e(G)}$.\label{eq:main2}
      \end{enumerate}
\end{theorem}

In the remainder of the introduction we provide a sketch of the proof
of Theorem~\ref{thm:mainresult}. In this paper, given functions $f=f(n)$ and $g=g(n)$, we write $f \ll g$ if $\lim_{n\to\infty} \frac{f(n)}{g(n)} = 0$.

First we show that for small values
of $p$, that is, as in Theorem~\ref{thm:mainresult}-\ref{resultado-parte1}, the number of Gallai 3-colourings in $G = G(n, p)$ approaches
the number of all 3-colourings of $E(G)$. This is natural
as colouring the edges of $G$ that do not belong to a triangle with
any colour (in $\{1, 2, 3\}$) and colouring the remaining edges with
colours in $\{1, 2\}$ always produces a Gallai colouring.  Therefore,
using $t(G)$ for the number of edges that belong to some triangles in
$G$, we have
\[ \Gallai(G) \geq 3^{e(G)-t(G)} \cdot 2^{t(G)}. \] 
For those ``small'' values of $p$
we consider three separate ranges of values for $p$. If $p \ll 1/n$,
then w.h.p.\ there are no triangles in $G(n,p)$ (this
follows by a classical result in~\cite{bollobas1987threshold}), so
$t(G) = 0$, making the assertion trivial. In the case
$A/n < p \ll 1/\sqrt{n}$, for any constant $A$, using Chebyshev’s
inequality, w.h.p.\ there are almost certainly far fewer
triangles than edges. Thus, if we remove one edge from each triangle
in $G$, the obtained graph will still have asymptotically the same
number of edges of $G$ but no triangles, and those remaining edges can
be coloured with no restriction, from which we conclude the
assertion. Finally, we show that there exists $c > 0$ such that if
$p < c\cdot n^{-1/2}$, but it is not the case that $p \ll 1/\sqrt{n}$,
we can use a similar strategy from the previous case but using a
result from \cite{DeMarco_2011} to estimate the number of triangles in
$G$ (instead of just using Chebyshev's inequality). As a matter of fact, in this last range Chebyshev's inequality alone does not yield the desired result, while in the other range the result from \cite{DeMarco_2011} does not hold.

In order to prove Theorem~\ref{thm:mainresult}-\ref{resultado-parte2},
since $p$ is large enough, there will be many
triangles in $G$, intuitively making it increasingly rare for a Gallai
colouring to use each of the 3 colours a significant number of
times. We use a multicolour version of the sparse Szemerédi's
Regularity Lemma to produce a ``reduced graph'' $R$ from the colouring of~$G$. Such lemma was first proved by Y.\,Kohayakawa and by V.\,Rödl (independently). See, for example, \cite{kohayakawa1997szemeredi,kohayakawa2003szemeredi}.

Instead of directly counting all Gallai $3$-colourings of $G$, we
restrict our counting to those Gallai $3$-colourings of $G$ that give
rise to a particular reduced graph $R$, and finally, we count the
possible number of reduced graphs. This counting method was originated
in \cite{alon2004number} (but using a standard dense version of the
Regularity Lemma). A crucial part of the final estimate requires an
inequality obtained in \cite{benevides2017edge}, which, to be applied
in our context, requires a sparse version of an embedding lemma. By
putting together all the estimates obtained during this counting
process, we conclude that the number of 3-colourings of Gallai in $G$
is bounded above by $2^{(1+\delta)e(G)}$.

In Section~\ref{sec:regularidade} we present the necessary lemmas from the regularity method that we will use in our proof. In Section~\ref{sec:proof} we present the detailed proof of Theorem~\ref{thm:mainresult}. Finally, in Section~\ref{sec:conclusion} we present some concluding remarks and open problems.

\section{The Regularity Method}
\label{sec:regularidade} 

For the upper range of $p$, the proof of our main result
(Theorem~\ref{thm:mainresult}) applies the Regularity Method in its
sparse version to a colored $G=G(n,p)$ \cite{kohayakawa2003szemeredi}. In this section we present the results we need in our proof
and we start by giving the definitions that will allow us to state the
Sparse Regularity Lemma.

Given $p\geq Cn^{1/2}$ and a graph $G$, let $X$ and $Y$ be disjoint
subsets of $V(G)$. The $p$-\emph{density} of $\{X, Y\}$ is defined as
$d_{G,p}(X,Y) = \frac{e_G(X,Y)}{p|X||Y|}$. We say that $\{X,Y\}$ is
\emph{$(\varepsilon, G, p)$-regular} if for every $X'\subseteq X$ and
$Y'\subseteq Y$ with $|X'|\geq\varepsilon|X|$ and
$|Y'|\geq\varepsilon|Y|$ we have
$\left|d_{G,p}(X',Y')-d_{G,p}(X,Y)\right| \leq \varepsilon$.  In the
particular case when $G=(X,Y;E)$ is a bipartite graph, with bipartition $X\cup Y$ and edge set $E$, and $\{X,Y\}$ is
$(\varepsilon, G, p)$-regular, then we say $G$ is
\emph{$(\varepsilon, p)$-regular}. Furthermore, if a bipartite
graph $G$ is $(\varepsilon, d)$-regular with $d=e_G(X,Y)/(|X|\,|Y|)$,
then we say that $G$ is $(\varepsilon)$-regular.

Given a set $V$ of $n$ vertices, a partition $\{V_0,\dots,V_k\}$ of $V$ is
$(\varepsilon,k)$\emph{-equitable} if $|V_0| \leq \varepsilon n$ and
$|V_1|=\dots=|V_k|$. Furthermore, given graphs $G_1,\dots,G_r$ on $V$, an $(\varepsilon,k)$-equitable partition is called $(\varepsilon,
G_1,\dots,G_r, p)$\emph{-regular} if all but at most $\varepsilon \binom{k}{2}$
pairs $\{V_i,V_j\}$, with $i, j \neq 0$, are $(\varepsilon,G_\ell, p)$-regular
for every $\ell\in[r]$. Finally, a graph $G$ with $n$ vertices is
$(\eta,p,D)$\emph{-upper-uniform} if for any disjoint subsets $U,W\subseteq
V(G)$ with $|U|,|W| \geq \eta n$ we have $|d_{G,p}(U,W)| \leq D$.  We will apply
the following ``multicolour'' version of the Sparse Regularity Lemma (see,
e.g.,~\cite[Theorem 4.4]{DuPr17}).

\begin{lemma}\label{lemma:multicolourSparseReg}
      For every $\varepsilon>0$ and $D \geq 1$ and a positive integer $r$, there
      exist $\eta$ and $M$ such that for every $0\leq p\leq 1$, if
      $G_1,\cdots,G_r$ are $(\eta, p, D)$-upper-uniform graphs on the same
      vertex set $V$, then there exists an
      $(\varepsilon,G_1,\cdots,G_r,p)$-regular partition of $V$ with $k+1$
      classes, where $1/\varepsilon \leq k \leq M$.
\end{lemma}

Given $G=G(n,p)$ and an $r$-colouring of $E(G)$, denote by $G_i$ the induced
subgraph of $G$ by the edges with colour $i$, for $i\in[r]$. It is not hard to
show that given any constant $0<\eta<1$, for the value of $p$ that we use ($p
\geq Cn^{-1/2} > 1/n$), with high probability $G(n,p)$ is
$(\eta,p,2)$-upper-uniform, which implies that $G_1,\dots,G_r$ are also
$(\eta,p,2)$-upper-uniform graphs. Then, we may apply
Lemma~\ref{lemma:multicolourSparseReg} to obtain an $(\varepsilon,k)$-equitable
$(\varepsilon,G_1,\cdots,G_r,p)$-regular partition $\mathcal{P} =
\{V_0,\cdots,V_k\}$ of $V(G)$. 

Given $\ell \in [r]$ and $\delta >0$, the \emph{reduced graph}
$R_\ell=R(G,\varepsilon,\delta,\ell,\mathcal{P})$ is the graph with vertex set
$[k]$, where $\{i,j\} \in E(R)$ if and only if $\{V_i,V_j\}$ is
$(\varepsilon,G_t,p)$-regular for every $t\in [r]$, and $d_{G_\ell,p}(V_i,V_j)
\geq \delta$.

Given a graph $H$ with vertex set $[k]$, $\varepsilon>0$, 
and integers $n$ and $m$, we define $\mathcal{G}(H,n,m,\varepsilon)$
as the class of graphs with vertex set $\bigcup_{i=1}^{k} V_i$, where
each $V_i$ is an independent set with $n$ vertices and, for every
$ij\in E(H)$, the pair $\{V_i,V_j\}$ is $(\varepsilon)$-regular with
exactly $m$ edges. Furthermore, $\mathcal{F}(H,n,m,\varepsilon)$ is
the subfamily of $\mathcal{G}(H,n,m,\varepsilon)$ whose elements do
not contain a \emph{transversal} copy of $H$, i.e., a copy of $H$
having one vertex in each $V_i$.  We will also use the notion of the
\emph{maximum $2$-density} of a graph $H$, defined (for $|V(H)|\ge 3$)
as
\[
      m_2(H) = \max\left\{\frac{e(F)-1}{|V(F)|-2} : F \subseteq H
            \text{ and } |V(F)| \geq 3\right\}.
\]

\begin{lemma}[Embedding lemma for subgraphs of random
            graphs~\cite{conlon2014klr}] \label{lemma:imersaoGNP}
      For any $0<\delta<1$ and any graph $H$, there exists
      $\varepsilon > 0$ such that for every $\beta>0$, there exists $C>0$
      such that for every $p\geq Cn^{\nicefrac{-1}{m_2(H)}}$, with high
      probability $G=G(n,p)$ does not contain a subgraph from
      $\mathcal{F}(H,n',\delta n'^2 p,\varepsilon)$ for every
      $n'>\beta n$.
\end{lemma}

The next result will allow us to show that $G$ contains a subgraph of
the family $\mathcal{G}(H,n,m,\varepsilon)$ with the appropriate
parameters (see, e.g., \cite[Lemma~4.3]{gerke_steger_2005}).

\begin{lemma}\label{lema2motaantiramsey2}
      For every $0<\varepsilon<1/6$, there exists $t$ such that every
      $(\varepsilon)$-regular bipartite graph $B=(V_1,V_2;E)$ contains a
      spanning subgraph $B'=(V_1,V_2;E')$ that is
      $(2\varepsilon)$-regular, containing exactly $\left|E'\right|=m$
      edges, for all $m$ satisfying
      $t|V_1 \cup V_2|\leq m \leq |E|$.
\end{lemma}

\section{Proof of the main result}
\label{sec:proof}

In this section we prove Theorem~\ref{thm:mainresult}, which follows directly
from Lemmas~\ref{lemma:main1} and~\ref{lemma:main2} below.
Lemmas~\ref{lemma:main1} and~\ref{lemma:main2} deal respectively with the
statements \ref{eq:main1} and~\ref{eq:main2} from Theorem~\ref{thm:mainresult}.

\begin{lemma}\label{lemma:main1}
      For every~$\delta>0$ there exists~$c$ such that the following holds for $G=G(n,p)$. If $p \leq c n^{-1/2}$, then w.h.p.\ we have $\Gallai(G)\geq 3^{(1-\delta)e(G)}$.
\end{lemma}

\begin{lemma}\label{lemma:main2} For every~$\delta>0$ there exists~$C$ such that
      the following holds for $G=G(n,p)$. If $p \geq C n^{-1/2}$, then w.h.p.\
      we have $\Gallai(G)\leq 2^{(1+\delta)e(G)}$.
\end{lemma}

\subsection{Lower range of $p$}
We start by analyzing the case where the random graph is very sparse. Before that,
let us state the following result of DeMarco and Kahn~\cite{DeMarco_2011}, which
will be useful for some ``very sparse'' range of $p$. Given a graph $G$, we
write $T(G)$ for the number of triangles in $G$.

\begin{lemma}[\cite{DeMarco_2011}, Theorem~1]
  \label{lemma:DeM_Kahn}
  For every $\xi > 0$, there is $D$ such that if $p > n^{-1}$, then
  for $G=G(n,p)$,
  $$
  \Pr[T(G) > (1+\xi)\mathbb{E}(T(G))] < \exp\left( -D \cdot\min \left(n^2 p^2
      \log(1/p), n^3 p^3 \right) \right).
  $$
\end{lemma}

\begin{proof}[Proof of Lemma~\ref{lemma:main1}]
  Fix $\delta > 0$ and set $c$ such that $6c^2 = \delta$. Let  $p \leq c
  n^{-1/2}$, and let $G=G(n,p)$ for sufficiently large $n$.      
  We shall prove that with high probability $\Gallai(G)\geq 3^{(1-\delta)e(G)}$.
  
  Recall that $t(G)$ stands for the number of edges of $G$ that
  belongs to some triangle. Obviously, we can assign any color to the
  remaining $e(G)-t(G)$ edges without risking creating a rainbow
  triangle. Thus, fixing two colours and using them to colour those
  $t(G)$ edges directly implies that
      \begin{equation}\label{eq:estimativaPPequeno}
            \Gallai(G)\geq 3^{e(G)-t(G)} 2^{t(G)}.
      \end{equation}
      For $p\ll n^{-1}$, it is well known that, with high probability, there are no triangles
      in~$G$, which implies $t(G)=0$. Consequently
      $\Gallai(G)= 3^{e(G)}\geq 3^{(1-\delta)e(G)}$. Thus, we may assume
      that $p \geq an^{-1}$ for some constant $a>0$.

      Next, consider the case where $p \ll n^{-1/2}$ but it does not hold that $p \ll 1/n$. Then, one can check that
      \[
            \Var(t(G)) \leq (n^4p^5 + n^3p^3)^2 \ll (pn^2)^2,
      \] which, combined with the fact that $e(G) = (1+o(1)){\binom{n}{2}}p$ with high probability and Chebyshev's inequality, implies that $t(G) =
      o(e(G))$. Together with \eqref{eq:estimativaPPequeno}, for
      sufficiently large $n$, this implies that with high probability
      \begin{equation*}
            \Gallai(G)\geq 3^{e(G)-t(G)} 2^{t(G)} \geq 3^{(1-\delta)e(G)}.
      \end{equation*}

      It remains only the case where $an^{-1/2} < p < cn^{-1/2}$, for some constant $a > 0$. Note that
      we cannot apply Chebyshev's inequality as in the previous case,
      because $p$ is not small enough. Instead, we use
      Lemma~\ref{lemma:DeM_Kahn}, with $\xi=1$, and the fact that $p \gg n^{-1}$ to
      conclude that with high probability the number of triangles in
      $G$ satisfies \[
            T(G) \leq 2p^3\binom{n}{3} \le n^3p^3 \leq 4 \cdot p^2n \cdot e(G),
      \]
      where the last inequality follows from the fact that w.h.p.\ $e(G) \geq pn^2/4$. In view of this, since each triangle
      contains three edges, the number of edges of $G$ that belong to
      triangles satisfies $t(G) \leq 3 \cdot T(G) \leq 12\cdot p^2n \cdot
      e(G)$. Therefore, as $\log_3(2) > 1/2$,
      \begin{equation*}
            \Gallai(G)
            \geq 3^{e(G)-t(G)} 2^{t(G)}
            \geq 3^{e(G)-t(G)/2}
            \geq 3^{(1-6np^2)e(G)}
            \geq 3^{(1-6c^2)e(G)}
            = 3^{(1-\delta)e(G)},
      \end{equation*}
      which finishes the proof of this lemma.
\end{proof}

Given a graph $G$, define the \emph{Gallai function}
$w : E(G)\to \{2,3\}$ as the function that assigns weight $2$ to edges
that belong to a triangle in $G$ and weight $3$ to all other
edges. Furthermore, with some abuse of notation, write $w(G)=\prod_{e\in E(G)}w(e)$. The following
result from \cite{benevides2017edge} will be useful.

\begin{lemma}[\cite{benevides2017edge}, Lemma 4.2] \label{lema4.2-HoppenBeneSampaio}
      Let $G$ be an $n$-vertex graph. If $w(G)$ is the Gallai function, then
      $w(G)\leq 2^{\binom{n}{2}}$.
\end{lemma}

\subsection{Upper range of $p$}

In this section we prove the following Lemma~\ref{lemma:main2}, which concludes
the proof of Theorem~\ref{thm:mainresult}.

\begin{proof}[Proof of Lemma~\ref{lemma:main2}]
      Fix $\delta>0$ and $\xi=\delta/4$. Furthermore, let $0<\alpha\leq 1/8$
      such that $6\alpha+H(2\alpha)\leq\delta/16$, where $H$ is the \emph{binary
      entropy function} $H: [0,1] \to \mathbb{R}$, defined as
      \[
            H(x)=-x\log_2(x) - (1-x)\log_2(1-x).
      \]
      This is possible because $H(x) \to 0$ as $x \to 0$ and it will be useful due to the well known estimate: 
      \[
             \binom{m}{x m} \leq 2^{H(x)m}.
      \] 

      Let $p \geq C n^{-1/2}$ and $G=G(n,p)$ for a
      sufficiently large $n$. We shall prove that with high probability $\Gallai(G)\leq 2^{(1+\delta)e(G)}$.

      Let $\alpha' = \frac{\alpha}{4}$ and let $\varepsilon$ be given by the
      embedding result (Lemma~\ref{lemma:imersaoGNP}) applied with $\alpha'$.
      Since for $0 < \varepsilon' \le \varepsilon$, every
      $(\varepsilon',p)$-regular pair is also $(\varepsilon,p)$-regular, we may
      and shall assume that $\varepsilon<\alpha/2$.  Furthermore, set
      $\varepsilon'= \varepsilon/2$.

      Now we apply the Sparse Regularity Lemma
      (Lemma~\ref{lemma:multicolourSparseReg}) with $\varepsilon'$ and the
      integers $m=1/\varepsilon'$ and $r = 3$ colours, obtaining constants
      $\eta$ and $M$. Finally, continuing the application of
      Lemma~\ref{lemma:imersaoGNP} with $\beta = (1-\varepsilon')/M$, we obtain
      a constant $C$. Let $p>C n^{-1/2}$ and $G=G(n,p)$ for a sufficiently large
      $n$. We shall prove that with high probability $\Gallai(G)\leq
      2^{(1+\delta)e(G)}$.     

      From now on we consider an arbitrary Gallai $3$-colouring of $E(G)$ (i.e.,
      with no rainbow triangles) with colours $1$, $2$ and $3$. As we discussed
      before, for $p \ge C n^{-1/2} \gg 1/n$, the graph $G$ is $(\eta,p,2)$-upper-uniform
      with high probability (it follows from Chernoff's inequality). Thus, from
      the Sparse Regularity Lemma (Lemma~\ref{lemma:multicolourSparseReg}),
      there is an $(\varepsilon,p)$-regular partition $\mathcal{P}=(V_0, \ldots,
      V_k)$ of $V(G)$ with $m\leq k\leq M$. Also from Chernoff's inequality,
      almost surely it holds that
      \begin{align*}
      &\big||E_G(V_i,V_j)| - |V_i||V_j|p\big| \leq \xi |V_i||V_j|p,\text{ and} \\
      &|E_G(V_i)|\leq(1+\xi)p{\binom{|V_i|}{2}} \leq (1+\xi)pn^2.
      \end{align*}

      Let $G_1$, $G_2$ and $G_3$, respectively, be the spanning subgraphs of
      $G$ induced by edges with colours $1$, $2$ and $3$. We consider the reduced graphs
      $R_\ell=R(G,\varepsilon',\alpha/4,\ell,\mathcal{P})$ for $\ell \in \{ 1, 2, 3 \}$, 
      which we recall denotes the graph with vertex set $[k]$, where $\{i,j\} \in E(R_\ell)$ if and only if
      $\{V_i,V_j\}$ is $(\varepsilon',G_\ell,p)$-regular, and
      $d_{G_\ell,p}(V_i,V_j) \geq \alpha/4$.  Finally, consider the graph $R$ with vertex set
      $[k]$ and edges $E(R_1)\cup E(R_2)\cup E(R_3)$.

      We shall prove that $\Gallai(G)\leq 2^{(1+\delta)e(G)}$. For that, we will
      give an upper bound on the number of 3-colourings of $G$ that specifically
      give rise to a fixed partition $\mathcal{P}$ and fixed reduced graphs
      $R_1, R_2$, and $R_3$. Finally, we will multiply this bound by the number
      of possible partitions and reduced graphs.

      \subsection*{Edges that do not contribute to the reduced graphs}
      We estimate the number of edges of $G$ that do not
      contribute to edges of the reduced graphs as follows.
      \begin{enumerate}
      \item[(i)] edges within some class $V_i$ for $1\leq i\leq k$;
      \item[(ii)] edges between pairs that are not $(\varepsilon')$-regular for
            some $1\leq \ell\leq 3$;
      \item[(iii)] edges between $(\varepsilon$)-regular pairs $\{V_i,V_j\}$ (
            $1\leq i< j\leq k$) with $d_{G_\ell,p}(V_i,V_j) < \alpha/4$ for every
            $1\leq \ell \leq 3$;
      \item[(iv)] edges with one endpoint in $V_0$.
      \end{enumerate}
      From a simple application of the Chernoff's inequality and the union
      bound, we have w.h.p.\ that
      $|E_G(V_i)|\leq(1+\xi)p{\binom{|V_i|}{2}} \leq
            \frac{(1+\xi)pn^2}{2k^2}$, from where we conclude that the number of
      edges within some class $V_i$ for $1\leq i\leq k$ is w.h.p.\ at most
      \begin{equation}\label{contagem1}
            k\left(\frac{(1+\xi)pn^2}{k^2}\right)
            \leq  (1+\xi)\varepsilon'pn^2
            \leq  \frac{(1+\xi)\alpha pn^2}{4}.
      \end{equation}
      One can also easily check that with high probability, for any
      $1\leq i< j\leq k$ it holds that
      $|E_G(V_i,V_j)| \leq (1 + \xi) |V_i||V_j|p \leq \frac{(1 +
                  \xi)pn^2}{k^2}$, which since there are at most $ \varepsilon'k^2$
      pairs that are not $(\varepsilon')$-regulars, implies that the number
      of edges between such pairs is at most
      \begin{equation}\label{contagem2}
            \varepsilon'k^2\left(\frac{(1+\xi)pn^2}{k^2}\right)
            \leq
            \frac{  (1+\xi)\alpha n^2p}{4}.
      \end{equation}
      We proceed by observing that w.h.p.\ there are no
      $\varepsilon'$-regular pairs $\{V_i,V_j\}$ with density smaller than
      $\alpha/4$ in $G_1$, $G_2$ and $G_3$ as $e_G(V_i,V_j) \leq
            3(\alpha/4)p|V_i||V_j| < p|V_i||V_j|/2$, which contradicts the
      fact that w.h.p. one can check that $e_G(V_i,V_j) \geq
            p|V_i||V_j|/2$.
      Finally, to bound the number of edges incidents to vertices in
      $V_0$, we use the facts that $|V_0|\leq \varepsilon'n$ and w.h.p.\ each
      vertex of $V_0$ has degree at most $(1+\xi)np$, which implies that the
      number of such edges is at most
      \begin{equation}\label{contagem4}
            \varepsilon'n (1+\xi)pn\leq(1+\xi)\varepsilon'n^2p\leq \frac{(1+\xi)\alpha n^2p}{4}.
      \end{equation}
      Putting inequalities \eqref{contagem1}--\eqref{contagem4} together, we
      conclude that the number of edges not related to edges of the reduced
      graphs $R_1$, $R_2$ and $R_2$ is at most
      $$
            (1+\xi)\alpha n^2p.
      $$
      Therefore, since w.h.p.\ $e(G)\leq (1 + \xi)pn^2/2$, the number of
      possible choices for these edges is w.h.p.\ at most
      \begin{align}
            \sum_{i=0}^{\lfloor(1+\xi)\alpha pn^2\rfloor}\binom{{(1+\xi)pn^2}/{2}}{i}
             & \leq 2\alpha pn^2\binom{{(1+\xi)pn^2}/{2}}{2\alpha pn^2}
            \leq 2^{\alpha pn^2}\binom{{pn^2}}{2\alpha pn^2}\nonumber            \\
             & \leq 2^{(pn^2)(\alpha+H(2\alpha))},\label{eq:arestasForaReduzido}
      \end{align}
      where the first inequality follows from
      $i \leq (1+\xi)\alpha pn^2 \leq 2\alpha pn^2 \leq (1+\xi)pn^2/2$, and
      the last one holds for sufficiently large $n$ where we used
      $\binom{x}{2\alpha x}\leq 2^{H(2\alpha)\cdot x}$.

      Recall that we want to bound the number of $3$-colourings of $E(G)$
      that give rise to the partition $\mathcal{P}$ and the reduced graphs
      $R_1, R_2$, and $R_3$. Any of the edges that do not contribute
      to pairs $\{V_i,V_j\}$ of $G$ that generate edges in these reduced
      graphs can receive any of the three possible colours. Therefore,
      the total number of possible colourings for all these (at most
      $(1+\xi)\alpha n^2p \leq 2\alpha n^2p$) edges is at most
      \begin{equation}
            3^{2{p}\alpha n^2}.\label{eq:arestasForaReduzido2}
      \end{equation}

      \subsection*{Edges that  contribute to reduced graphs}

      Now we count the set of edges that actually gives rise to the reduced
      graph $R$. Given an edge $ij$ of $R$ (recall that $R$ is the graph
      with vertex set $[k]$ and edges $E(R_1)\cup E(R_2)\cup E(R_3)$),
      denote by $s_{i,j} \in \{1, 2, 3\}$ the number of reduced graphs among
      $R_1, R_2$, and $R_3$ in which $ij$ is an edge of it. We first need to
      obtain an upper bound to the number of
      \begin{itemize}
            \item edges between $(\varepsilon')$-regular pairs $\{V_i,V_j\}$ ($1\leq i< j\leq k$) with $d_{H_\ell,p}(V_i,V_j) \geq \alpha/4$ for some
                  $1\leq \ell \leq 3$.
      \end{itemize}
      Let $ij \in E(R)$ and note that if $ij$ does not belong to $R_s$, then
      $e_{H_s}(V_i,V_j) \leq (\alpha/4)p(n/k)^2$. On the other hand, if $ij$
      is an edge of $R_s$, then we know that w.h.p.\ we have
      $e_{H_s}(V_i,V_j) \leq (1+\xi)pn^2/k^2$. Therefore, the
      number of possible ways to colour the edges between
      $V_i$ and $V_j$ in $G$ is at most
      $$
            s_{i,j}^{{p(1+\xi)}n^2/k^2}(3-s_{i,j})^{(\alpha/4)p(n/k)^2}
            \leq
            s_{i,j}^{{p(1+\xi)}n^2/k^2}3^{(\alpha/4)p(n/k)^2}.
      $$
      Now, for $s\in \{1, 2, 3\}$, let $E_s$ be the set of edges in the reduced graph that appear in
      exactly $s$ reduced graphs, and $e_s=|E_s|$. Thus, the
      total number of colourings that generate $R$ is at most
      $$
            \left(1^{e_1}2^{e_2}3^{e_3}\right)^{{p(1+\xi)}(n/k)^2}3^{\alpha p(n/k)^2}.
      $$
      This together with~\eqref{eq:arestasForaReduzido}
      and~\eqref{eq:arestasForaReduzido2} implies that the potential number
      of 3-colourings of $G$ that can generate this particular partition
      $\mathcal{P}$ of vertices and the reduced graphs $R_1$, $R_2$ and
      $R_3$ is at most
      \begin{align}
                 & 2^{(pn^2)(\alpha+H(2\alpha))}3^{2{p}\alpha n^2}3^{\alpha
            p(n/k)^2}(1^{e_1}2^{e_2}3^{e_3})^{{p(1+\xi)}(n/k)^2}           \nonumber \\ 
            &\leq  2^{(pn^2)(\alpha+H(2\alpha))}2^{5\alpha
            pn^2}(1^{e_1}2^{e_2}3^{e_3})^{{p(1+\xi)}(n/k)^2}                \nonumber \\
            &= 2^{(6\alpha+H(2\alpha))
                        pn^2}(1^{e_1}2^{e_2}3^{e_3})^{{p(1+\xi)}(n/k)^2}.
            \label{eq:contagemCrua}
      \end{align}
      We will show now that
      $(1^{e_1}2^{e_2}3^{e_3})^{{p(1+\xi)}(n/k)^2} \leq
            2^{{p(1+\xi)}n^2/2}$. 

            For that, consider the function $f$ from edges of $R$ to $\{1, 2, 3\}$ that assigns
      value $i$ to the edges belonging to $E_i$ for every $i\in \{1, 2, 3\}$. To simplify our analysis,
      define $R'$ as the subgraph of $R$ obtained by deleting edges of $E_1$, i.e., edges that appear only in one of the reduced
      graphs $R_1$, $R_2$, $R_3$. Note that
      \[
            1^{e_1}2^{e_2}3^{e_3} = 2^{e_2}3^{e_3} = \prod_{e\in E(R')}f(e).
      \]

      The definition of $R'$ together with the embedding lemma for $G(n,p)$ (Lemma~\ref{lemma:imersaoGNP}) implies the following.
      \begin{claim}\label{afirmacao1} Almost surely (with respect to the choice of $G$), no triangle in $R'$ contains an edge of weight 3.
      \end{claim}
      We will postpone the proof of this claim. For now, note that it implies $f(e) \le w(e)$ for every edge in $R'$, where $w$ is the Gallai function of the graph $R'$, as defined before Lemma \ref{lema4.2-HoppenBeneSampaio}. Therefore, $\prod_{e\in E(R')}f(e) \le w(R')$. 
      And by Lemma \ref{lema4.2-HoppenBeneSampaio}, $ w(R') \leq 2^{\binom{k}{2}}\leq
            2^{k^2/2}$. These imply that, in fact,
      \begin{equation}\label{eq:arestasReduzido1}
            (1^{e_1}2^{e_2}3^{e_3})^{{p(1+\xi)}(n/k)^2} \leq 2^{{p(1+\xi)}n^2/2}.
      \end{equation}

      To finish our proof we calculate the number of possible distinct
      partitions and reduced graphs. The number of vertex partitions is
      at most $\sum_{k=m}^M k^n$, while, for a fixed partition with $k$ parts, the number
      of distinct triples of reduced graphs $(R_1, R_2, R_3)$ is at most
      $8^{\binom{k}{2}}$, since for each of the $\binom{k}{2}$ pairs of classes in the partition there are 8 possible ways to decide to which of $R_1$, $R_2$, $R_3$ it belongs. Altogether, the number of possible partitions and reduced graphs is at most \[
            M^{n+1}\,8^{\binom{M}{2}} < 2^{\delta pn^2 /16},
      \]
      as $M$ and $\delta$ are positive constants and $pn^2 \gg n$.

      Finally, by Equation~\eqref{eq:contagemCrua} and
      its developments, and due to our initial choice of $\alpha$, taking
      $n$ sufficiently large to satisfy all asymptotic assertions we made,
      with high probability we have the following:
      \begin{align}
            \Gallai(G) & \leq \left[M^{n+1}\cdot2^{3\binom{M}{2}}\right]\cdot \left[2^{(6\alpha +H(2\alpha))pn^2}\right]\cdot \left[2^{{(1+\xi)}n^2p/2}\right]
            \nonumber                                                                                                                         \\
                       & \leq \left[2^{\delta pn^2 /16}\right]\cdot\left[2^{\delta pn^2/16}\right]\cdot\left[2^{{(1+\delta/4)p n^2}/2}\right]
            \nonumber                                                                                                                         \\
                       & = 2^{\left(pn^2/2\right)\left[\delta/8+\delta/8+(1+\delta/4)\right]}
            \nonumber                                                                                                                         \\
                       & = 2^{\left(p{n^2/2}\right)(1+\nicefrac{\delta}{2})}
            \nonumber                                                                                                                         \\
                       &\leq 2^{(1+\delta)e(G)}\label{eq:final2},
      \end{align}
      where the last inequality holds with high probability due to Chernoff's inequality. In fact, taking a constant $\lambda > 0$ such that 
      \[
            1 - \lambda > \frac{1+\nicefrac{\delta}{2}}{1+\delta},
      \]
      it follows that with high probability
      \[
            e(G) > (1-\lambda)p\binom{n}{2} > \frac{(1+\nicefrac{\delta}{2})}{(1+\delta)}\frac{pn^2}{2},
      \]
      for $n$ sufficiently large. This finishes the proof of Lemma~\ref{lemma:main2}
\end{proof}

\begin{proof}[{\bf Proof of Claim \ref{afirmacao1}}]
Consider a Gallai $3$-colouring of $G = G(n,p)$, and the construction of $R'$
as in the proof of Lemma~\ref{lemma:main2}. Suppose for a contradiction that there is a triangle $S$ in $R'$ that contains an edge of weight~$3$. Without loss of generality, let $S$ be a triangle in $R'$ with vertices $u, v, w$ and edges $uv, vw, wu$ with weights $3,2,2$ respectively. (Recall that in \(R'\), there
are no edges of weight 1 and if the weights of the edges of $R'$ were $(3, 3, 2)$ or $(3, 3, 3)$ we can simply ignore some of the colours). 

We will use the embedding lemma for subgraphs of $G = G(n,p)$
(Lemma~\ref{lemma:imersaoGNP}) to show that the induced subgraph of $G$ on
$V(S)$ contains a rainbow triangle, which contradicts the fact that $G$ is coloured with a Gallai colouring.

The fact that the edge $uv$ has weight $3$ means that it belongs to all 3
reduced graphs $R_1, R_2, R_3$ and the fact that $vw$ and $wu$ have weight 2
means that each of them belongs to two reduced graphs. Assume, without loss of
generality, that $wu \in E(R_3)$. Since $vw$ belongs to 2 of the reduced graphs,
at least one of them is different from $R_3$. Suppose, again without loss of
generality, that $vw \in E(R_2)$. Furthermore, since $uv$ belongs to all three
reduced graphs, in particular, it belongs to $E(R_1)$. Thus, we have $uv \in
E(R_1)$, $vw \in E(R_2)$ and $wu \in E(R_3)$. Let $E_1$ be the set of edges of
colour 1 in $G[V_u,V_v]$, $E_2$ the set of edges of colour 2 in $G[V_v,V_w]$ and
$E_3$ the set of edges of colour 3 in $G[V_w,V_u]$.  Assume w.l.o.g.\;that
$|E_1| = \min(|E_1|, |E_2|, |E_3|)$. We now consider \(G_S\), the subgraph of
\(G\) induced by the vertex set \(V_{u}\cup V_v \cup V_w\) and edge set \(E_1
\cup E_2 \cup E_3\).

By Lemma \ref{lema2motaantiramsey2}, applied to each of the induced bipartite
graphs $G[V_i,V_j]$ for $i,j \in \{u,v,w\}$, $i \neq j$, with parameter
\(\varepsilon'\), we obtain constants $t_1$, $t_2$ and $t_3$. Assume that $t_1 =
\max\{t_1,t_2,t_3\}$. From the conclusion of Lemma~\ref{lema2motaantiramsey2},
we know that there exists a spanning tripartite subgraph $J$ of \(G_S\) that is
\((\varepsilon)\)-regular (since $\varepsilon' = \varepsilon/2$) with exactly
\(m\) edges between each pair of parts, for every \(m\) satisfying
\begin{equation}
\label{eq:t1}
t_1\cdot2\left(\frac{n-|V_0|}{k}\right) \le m \le |E_1|.
\end{equation} 

Put \(m=\alpha'\left(\frac{n-|V_0|}{k}\right)^{2} p\) and let us show that we have \(t_1\cdot2\left(\frac{n-|V_0|}{k}\right) \le m \le |E_1|\). Note that, for $1\leq \ell\leq 3$, by the definition of the reduced graph $R_\ell = R(G, \varepsilon', \alpha/4, \ell, \mathcal{P})$, since \(|V_i|=(n-|V_0|)/k\) for every possible $i$, we have that $|E_\ell| \geq \frac{\alpha p}{4}  \left(\frac{n-|V_0|}{k}\right)^{\!2}$. 
Therefore,
 \[
m =  \alpha'p \left(\frac{n-|V_0|}{k}\right)^{\!2} = \frac{\alpha p}{4} \left(\frac{n-|V_0|}{k}\right)^{\!2} \leq |E_1|.
 \]  
Also, since $n$ is sufficiently large,
 \[
m =  \alpha'p \left(\frac{n-|V_0|}{k}\right)^{\!2} \geq t_1\cdot2\left(\frac{n-|V_0|}{k}\right),
 \]
 which confirms that inequalities in~\eqref{eq:t1} hold for our choice of $m$.

Note that the subgraph $J$ is a spanning tripartite subgraph of $G_S$ with parts $V_u, V_v, V_w$ and exactly $m$ edges between each pair of parts, where each such pair is $(\varepsilon)$-regular. Therefore, $J$ is a member of the family \(\mathcal{G} =
\mathcal{G}\left(K_3, \frac{n-|V_0|}{k}, m, \varepsilon\right)\).

From Lemma \ref{lemma:imersaoGNP}, we know that almost surely $J$ is not a member of the family \(\mathcal{F} = \mathcal{F}\left(K_3,\frac{n-|V_0|}{k}, m, \varepsilon\right)\). This says that the coloring of $G$ contains a rainbow triangle, which is a contradiction because we started with a Gallai colouring. This finishes the proof of Claim~\ref{afirmacao1}.
\end{proof}

\section{Concluding remarks} \label{sec:conclusion}
We conjecture that the natural generalization of Theorem~\ref{thm:mainresult} for $k$-colourings of Gallai is true, with the same threshold function $O(n^{-1/2})$. We also believe it is possible to improve the estimates in such a theorem to something of the form $\Gallai_k(G) \le (1+\delta)2^{|E(G)|}$ for $G = G(n,p)$ with $p \ge Cn^{-1/2}$ when $k$ is constant and $n$ is sufficiently large.

\bibliographystyle{amsplain}
\bibliography{bibliography.bib}

\section{Appendix: Concentration inequalities}

We apply Chebyshev’s inequality and Chernoff’s inequality in the following forms.

\begin{lemma}[Chebyshev inequality]\label{lem:chebyshev}
      Let $X$ be a non-negative random variable and $\lambda$ be positive. Then
      \begin{equation*}
            \Pr\!\big(|X-\mathbb{E}(X)|>\lambda) \leq \frac{\operatorname{Var}(X)}{\lambda^2},
      \end{equation*}
\end{lemma}

\begin{lemma}[Chernoff inequality]\label{chernoff}
      Let $X$ be a random variable with binomial distribution $B(n,p)$ and let $\mu=\mathbb{E}(X)$. Then, for every $ 0\leq\xi\leq1 $ we have:
      \begin{align*}
            \Pr\!\big(X \leq \left(1-\xi\right)\mu\big) \leq \exp\left(\frac{-\xi^2\mu}{2}\right); \\
            \Pr\!\big(X \geq \left(1+\xi\right)\mu\big) \leq \exp\left(\frac{-\xi^2\mu}{2+\xi}\right).
      \end{align*}
\end{lemma}

\begin{corollary}\label{cor:chernoffAplicacao}
      With the assumptions of Lemma \ref{chernoff}, if $\xi$ is fixed and
      $\mu \to \infty$ as $n\to \infty$, then with high probability we have
      \begin{equation*}
            \big|X - \mu \big| \leq \xi \mu.
      \end{equation*}
\end{corollary}
\end{document}